\documentclass[11pt]{article}

\usepackage{amssymb}
\usepackage[centertags]{amsmath}
\usepackage{amsthm}

\newtheorem{theorem}{Theorem}

\newtheorem{corollary}{Corollary}

\newtheorem{example}{Example}
\newtheorem{definition}{Definition}
\newtheorem{remark}{Remark}

\newcommand{\Graph}{\operatorname{Graph}}

\newcommand{\R}{\mathbb{R}}
\newcommand{\skal}[2]{\langle #1,#2 \rangle}

\title{Weak Cyclic Monotonicity and\\Existence of Solutions of Differential Inclusions }
\author{Elza Farkhi\thanks{School of Mathematical Sciences, Sackler Faculty of Exact
Sciences, Tel Aviv University, 69978 Tel Aviv, Israel, email: 
elza@post.tau.ac.il} 
}

\begin{document}

\maketitle


\begin{abstract}
The notion of weak cyclic monotonicity of set-valued maps generalizing the cyclyc
monotonicity is introduced. The 
existence of solutions of differential inclusions with compact, upper semi-continuous,  
not necessarily convex right-hand sides 
in $\mathbb{R}^n$ is proved for weakly cyclic monotone 
right-hand sides.
 
\end{abstract}
\begin{quote}
\end{quote}
\begin{quote}
\end{quote}

\section{Introduction} 

We investigate the differential inclusion
\begin{equation}\label{1} 
 \dot x(t)\in F(x(t)),\quad x(0)= x_0\in \mathbb{R}^n,\quad t\in I=[0,T], 
\end{equation} 
where the set-valued mapping $F$ is upper semi-continuous, with non-empty compact, 
not necessarily convex values in $\mathbb{R}^n$. 

Among the relevant results on existence of solutions of such inclusions,
we should first mention the case of upper semi-continuous maps with convex compact values
(see e.g. \cite[Chapter 2]{D}). 
In particular, the classical existence theorem for right-hand side
which is the negation of a maximal monotone 
map  \cite[Sec.~3.2, Theorem~1]{AubCel:84}.
Maximal monotone set-valued maps are almost everywhere single-valued \cite{Zar:73,Kend:74},
 and it is easy to see that at the points where they are not single-valued, they
are upper semi-continuous with convex values.
Other important existence results for differential inclusions with non-convex right-hand sides
 are the results of Filippov \cite{FIL:zametki_71} in finite dimensions and De Blasi, Pianigiani
\cite{DB_P:91} in Banach spaces for  continuous 
$F$. The continuity may replace the monotonicity and convexity, as is shown there.

An important existence result for upper semi-continuous maps,
with not necessarily convex values, is established in \cite{BreCelCol:89},
where the condition of cyclic monotonicity is imposed. 
We remind that the values of a cyclic monotone map are subsets of the values of the subdifferential
of a proper convex function \cite{Rock:ConvAnal}.
Generally,  cyclic monotonicity is stronger than just monotonicity, but  
in $\R^1$ they are equivalent.

In \cite{KrRibTsa:07} the notion of colliding on the set of discontinuities of $F$ is investigated 
and geometric conditions to avoid it or to escape of this set are studied.

Existence of solutions is established in \cite{FarDonBai:14} 
uder  a weak  componentwise monotonicity condition and a standard growth condition.
This componentwise monotonicity condition is equivalent to 
 the strengthened one-sided Lipschitz (SOSL) condition \cite{LemVel:98} with constant zero
imposed on the negation of $F$, $-F(\cdot)$. 
In one dimension the latter condition is equivalent to
the  one-sided Lipschitz (OSL) property  in the sense of \cite{DonFar:98} of \ $-F$ with constant zero.
In the  proof a limiting procedure for the Euler polygons and their velocities is used. 
The method of Euler polygonal approximations is a widely exploited tool for deriving existence, 
approximation, well-posedness and stability properties of solutions of
differential equations and inclusions (see e.g. \cite{ADonFar:89,DonFar:98,Veliov:97}), 
as well as of variational and control problems \cite{DonLem:92, 
Mord:93,Mord:bookII}. A typical proof of the existence of solutions
in the upper-semicontinuous case  uses the Arzela-Ascoli theorem applied
to the polygonal solutions and/or Mazur's weak closure theorem for the velocities. The
limit of a convergent sequence of Euler approximants is then a solution of the inclusion with convexified right-hand
side. In \cite{FarDonBai:14} the existence of solutions without convexity is achieved 
deriving monotonicity in the time of each coordinate of the velocities from the
weak  componentwise monotonicity condition. 
The Helly's selection principle (see e.g. \cite[Chap.~10]{KF}) is applied there instead of the Mazur's 
weak closure theorem to obtain strong compactness of the set of velocities.

Here we formulate a condition called ``weak cyclic monotonicity'', which generalizes 
the cyclic monotonicity 
and we prove the existence of solutions under this condition. 
In  the proof we construct a cyclic monotone subinclusion of the given inclusion
and apply the existence theorem of \cite{BreCelCol:89}.
Let us note that in general 
the weak cyclic monotonicity condition is stronger than the weak monotonicity condition.
The question whether one can prove existence for weakly monotone maps
(maps having OSL negation with constant zero)
 in more than one dimension is still open.


\section{Weak cyclic monotonicity}

First we introduce some notation. For notions used in the paper, but not 
explicitly defined here we refer the reader to \cite{D, Rock:ConvAnal, AubCel:84} . 

Let $u,v\in \mathbb{R}^n$. We denote by $\skal{u}{v}$ the scalar product  of the vectors $u,v$ 
and by $|v|$ the Euclidean norm of $v$.
For a bounded set $A\subset\R^n$,
we denote $\|A\|=sup\{|a|:a\in A\}$. The support function of the set $A$ in direction $v$
is denoted by $\delta^*(v,A)=\sup_{a\in A} \skal{v}{a}$.
The classical  convex subdifferential of a convex function $f$ at the point $x$ is denoted by $\partial f(x)$.

We use the fact that the set-valued map $F$ from $\mathbb{R}^n$ to a compact set in $\mathbb{R}^n$
is upper semicontinuous (USC) iff its graph is closed \cite[Section 1.1]{AubCel:84}.

We recall that the map $F:\R^n\rightrightarrows \mathbb{R}^n$ is monotone if for every $x,y\in\mathbb{R}^n,v_x\in F(x),v_y\in F(y)$,
\begin{equation}\label{mon} 
 \left<x- y, v_x- v_y\right> \geq 0 . 
\end{equation}
We call $F$  \textbf{\emph{weakly monotone}} if for every $x,y\in\mathbb{R}^n$ and every 
$v_x\in F(x)$ \emph{there exists} $v_y\in F(y)$ such that \eqref{mon} holds.

We also recall that the (monotone) set-valued function $F: \mathbb{R}^n\rightrightarrows \mathbb{R}^n$
is called 
\textbf{\emph{cyclic monotone}} if for every natural $m$, every $x_0,x_1,...,x_m\in\mathbb{R}^n$,
and every $v_i\in F(x_i), \ i=0,...m$,
\begin{equation}
\label{mc} 
 \skal{x_m- x_0}{v_m}\geq \sum_{i=1}^{m} \skal{x_i- x_{i-1}}{v_{i-1}}.
\end{equation}
It is well-known that monotone, hence all cyclyc monotone maps, are almost everywhere single-valued 
(in the sense of measure \cite{Zar:73} and in Baire category sense \cite{Kend:74}). 
The weaker property we define below 
may hold for everywhere non-single-valued maps.

In the next definitions and in all considerations below we fix the point $x_0$ to be the one of  \eqref{1}
and also fix an arbitrary $v_0\in F(x_0)$.
\begin{definition}
\label{Def_WECM_seq} 
The finite 
sequence 
$\{(x_i,v_i)\}_{i=0}^k\subset Graph(F)$ is a \textbf{cyclic monotone sequence (CM sequence)}
if \eqref{mc} holds for every $m=1,2,...k$. 
The one-term sequence $\{(x_0,v_0)\}\subset Graph(F)$ is  a CM sequence 
 by definition. 
\end{definition}
It is clear from this definition that if one cuts several last terms
of the CM sequence $\{(x_i,v_i)\}_{i=0}^k$, the shorter
sequence $\{(x_i,v_i)\}_{i=0}^m$ for $m\le k$ is also CM.

Now we give a recursive definition of a weakly cyclic monotone map
based on cyclic monotone sequences on its graph.
\begin{definition}
\label{Def_WECM} 
The 
set-valued map $F: \mathbb{R}^n\rightrightarrows \mathbb{R}^n$
is called \textbf{weakly cyclic monotone (WCM)} 
if 
for every natural $m$, any CM 
sequence $\{(x_i,v_i)\}_{i=0}^{m-1}\subset Graph(F)$ 
and  every $x_{m}\in\mathbb{R}^n$ 
there exists $v_{m}\in F(x_{m})$ such that the sequence
$\{(x_i,v_i)\}_{i=0}^{m}$ is a CM 
sequence.
\end{definition}
In other words, $F$ is weakly cyclic monotone if every CM sequence  may be continued 
on the graph of $F$,  
preserving its cyclic monotonicity. 

Clearly, the weak cyclic monotonicity generalizes the cyclic monotonicity. 
On the other side, the weak cyclic monotonicity of set-valued maps is stronger than the weak monotonicity defined above.

\begin{remark}
\label{cond_WCM2}
The following condition 
implies the WCM condition:
\begin{equation}
\label{WMC2}
\delta^*(x_m- x_0,F(x_m))\geq \sum_{i=1}^{m} \delta^*(x_i- x_{i-1},F(x_{i-1})),
\end{equation}
for every finite sequence $\{x_i\}_{i=0}^m$. Indeed,
to see that 
the  condition \eqref{WMC2} implies weak cyclic monotonicity, 
we take an arbitrary CM 
sequence $\{(x_i,v_i)\}_{i=0}^{m-1}\subset Graph(F)$ 
and   $x_{m}\in\mathbb{R}^n$ 
and chose $v_{m}\in F(x_{m})$ such that 
$\skal{x_m- x_0}{v_m}= \delta^*(x_m- x_0,F(x_m))$. Then by \eqref{WMC2}, 
$$\skal{x_m- x_0}{v_m}= \delta^*(x_m- x_0,F(x_m))\geq \sum_{i=1}^{m} \delta^*(x_i- x_{i-1},F(x_{i-1}))
\geq \sum_{i=1}^{m} \skal{x_i- x_{i-1}}{v_{i-1}},$$
which implies that the sequence
$\{(x_i,v_i)\}_{i=0}^{m}$ is a CM 
sequence.

On the other hand, 
the condition \eqref{WMC2} seems stronger than the WCM condition,
since in Definition \ref{Def_WECM} there exists $v_m$ satisfying \eqref{mc} 
only for a CM sequence $\{x_i,v_i\}_{i=1}^{m-1}$, while \eqref{WMC2}
means the existence of $v_m$ satisfying \eqref{mc} for every sequence 
$\{x_i,v_i\}_{i=1}^{m-1}$ in the graph of $F$.
\end{remark}


%



\begin{remark}
\label{rem_wcm2}
Given a CM sequence $\{(x_i,v_i)\}_{i=0}^{m-1}$ and $x_m\in\R^n$, 
if one can choose $v_m\in F(x_m)$ satisfying
\begin{equation}
\label{wcm2}
\skal{x_m- x_0}{v_m-v_{m-1}}\ge 0,
\end{equation}
then the sequence $\{(x_i,v_i)\}_{i=0}^{m}$ is CM.

Indeed, \eqref{wcm2} may be written as
\[
\skal{x_m- x_0}{v_m} \ge \skal{x_m- x_{m-1}}{v_{m-1}} + \skal{x_{m-1}- x_0}{v_{m-1}}.
\]
From this, one can prove 
\eqref{mc} 
using a simple inductive argument.
\end{remark}
\begin{example}
Here we give a simple example of a WCM multifunction which is not cyclic monotone.
Given a compact set $A$, the constant map $F:\R^n\rightrightarrows\R^n$ defined by 
$F(x)\equiv A,\ x\in\R^n$, is WCM. Indeed, for a given CM 
sequence $\{(x_i,v_i)\}_{i=0}^{m-1}$ 
and for given $x_m\in\mathbb{R}^n$ 
one can choose $v_{m}=v_{m-1}\in A$ so as to fulfill \eqref{wcm2}.
If $A$ is not a singleton, then $F$ is not cyclic monotone since
cyclic monotone (and all monotone) maps are a.e. single-valued, as we have mentioned above.
\end{example}

\section{The Main Result}

We impose the following assumptions in order to prove the existence of solutions: 
\vskip 0.5em 
\textbf{A1.} $F: \mathbb{R}^n\rightrightarrows \mathbb{R}^n$ has
compact, nonempty values 
and is upper semicontinuous.

\vskip 0.5em 

\textbf{A2.} The map $F$ is weakly cyclic monotone.
\begin{remark}
\label{rem_USC_cls_gph}
We recall that the map $F: \mathbb{R}^n\rightrightarrows \mathbb{R}^n$  is locally
bounded if for any $x_0\in\R^n$ there is $\varepsilon>0$ such that
$\{y\in F(x): |x-x_0|<\varepsilon\}$ is bounded. It is well-known (see e.g. \cite{Smirnov:02})
that a locally bounded map $F: \mathbb{R}^n\rightrightarrows \mathbb{R}^n$ is
upper semicontinuous if and only if it has closed graph. 
\end{remark}

The main result of this paper is:
\begin{theorem}\label{ThM}
 Under the conditions \textbf{A1, A2} 
there exists $T>0$ such that 
 the differential inclusion (\ref{1}) has a 
solution on $[0,T]$. 
\end{theorem} 
 To proof the theorem, we use the following result proved in \cite{BreCelCol:89}.
\begin{theorem}
\label{th_BCC}
 If $F: \mathbb{R}^n\rightrightarrows \mathbb{R}^n$ has non-empty compact images
and is upper semicontinuous and cyclic monotone, then there exists $T>0$ such that 
 the differential inclusion (\ref{1}) has a 
solution on $[0,T]$. 
\end{theorem}

We reduce our considerations to the last theorem by the following claim.
\begin{theorem}\label{l1}
 If $F: \mathbb{R}^n\rightrightarrows \mathbb{R}^n$, 
satisfies {\bf A1,A2},  then for every fixed $x_0\in\mathbb{R}^n,v_0\in F(x_0)$ 
the following map is upper semicontinuous and cyclic monotone, with nonempty
compact images:
\begin{equation}
\label{def_G}
G(x)=\left\{ v\in F(x) : \skal{x-x_0}{v} \ge \sup\Bigl\{ \skal{x-x_m}{v_m}+ \sum_{i=1}^m \skal{x_i-x_{i-1}}{v_{i-1}} 
																													\Bigr\} \right\},
\end{equation}
where the supremum is taken on all natural $m$ and all CM 
sequences $\{(x_i,v_i)\}_{i=0}^m$. 

\end{theorem}
\begin{proof} 
Let $x_0,v_0\in F(x_0)$ be fixed as above. 
 For a given finite CM sequence $S=\{(x_i,v_i)\}_{i=0}^m$
and for $x\in\R^n$ we denote 
\[
g(S,x)=\skal{x-x_m}{v_m}+\sum_{i=1}^m \skal{x_i-x_{i-1}}{v_{i-1}}.
\]
Define the function
\begin{equation}
\label{eq_def_g}
g(x)=\sup_{S}g(S,x),
\end{equation}
where the supremum is taken on all $m\in \mathbb{N}$ and all finite CM 
sequences $S=\{(x_i, v_i)\}_{i=0}^m \subset \Graph(G)$.
Clearly, $g$ is convex, as a supremum of affine functions. Further, 
\[
g(x_0)=\sup_{S}
       \Bigl\{ \skal{x_0-x_m}{v_m}+\sum_{i=1}^m \skal{x_i-x_{i-1}}{v_{i-1}} \Bigr\},
\]
and 
using \eqref{mc} we get that $g(x_0)\le 0$. Taking $m=0$, we easily see
that $g(x_0)=0$, which implies that $g$ is proper.

Next, we prove that $G(x)$ is non-empty for every $x\in\mathbb{R}^n$.
Let us note that in the above notation
\[
G(x)=\left\{ v\in F(x) : \skal{x-x_0}{v} \ge  g(x)  \right\}.
\]
Given a finite CM sequence $S=\{(x_i,v_i)\}_{i=0}^m$ and 
setting $x_{m+1}=x$, by \textbf{A2} there is $v_{m+1}\in F(x)$, which
we denote by $v(S,x)$, such that
\[
\skal{x-x_0}{v(S,x)}\ge g(S,x). 
\]
Since $g(x)$ is a supremum,
we can choose a sequence $\{S^k\}_{k=1}^\infty$ such that 
\[
g(x)=\lim_{k\to\infty}g(S^k,x)\le \liminf_{k\to\infty}\skal{x-x_0}{v(S^k,x)}.
\]
Since $F(x)$ is compact, the limit of every convergent subsequence of $\{v(S^k,x)\}_{k=1}^\infty$ is in $F(x)$.
Denoting a limit of a converging subsequence of $\{v(S^k,x)\}_{k=1}^\infty$  by $v(x)\in F(x)$,
we get that $g(x)\le \skal{x-x_0}{v(x)}$, thus $v(x)\in G(x)$.
 Thus, \  $G(x)\neq\emptyset$ for any $x\in\mathbb{R}^n$.
Clearly, $G(x)$ is compact as a closed subset of the compact $F(x)$.

Now, by \cite[Proposition 1.1.3]{AubCel:84},
every upper semicontinuous map with compact values defined on a compact domain has a compact range,
therefore $F$ is locally bounded. Hence $G$ is locally bounded as a sub-mapping of $F$ and
according to Remark \ref{rem_USC_cls_gph}, in order to prove that $G$ is upper semicontinuous, 
it is enough to show that its graph is closed.
Take the sequences $x^k\to x^\infty$ and $v^k\to v^\infty$ with $v^k\in G(x^k)$.
Passing to the limit for $k\to\infty$ in the inequalities of \eqref{def_G} and taking in account that $F$ has closed graph,
we get that $v^\infty\in G(x^\infty)$, hence $G$ is USC.

To show that $G$ is cyclic monotone, we use the criterion that a map $G$ is cyclic monotone 
if and only if $G(x)$ is a subset of the subdifferential of a proper convex function \cite{Rock:ConvAnal}.
Our argument is similar to the one in Theorem 24.8 in \cite{Rock:ConvAnal}.

We now prove that $G(x)\subset \partial g(x)$ for any point $x\in \mathbb{R}^n$.
 For this, it is sufficient to show that for all $(x,v)\in\Graph(G)$ and for any $y\in\R^n$ and $\alpha<g(x)$,
\[
g(y)-\alpha >\skal{y-x}{v}.
\]
Indeed, since $g(x)$ is a supremum and $\alpha<g(x)$, there is
a finite CM sequence $S=\{(x_i,v_i)\}_{i=0}^k$,  
such that
\begin{equation}
\label{ineq_alpha}
\alpha < g(S,x)=\skal{x-x_k}{v_k}+\sum_{i=1}^k \skal{x_i-x_{i-1}}{v_{i-1}}.
\end{equation}
Now, adding the point $x_{k+1}=x,v_{k+1}=v\in G(x)$ to the sequence $S$, we obtain 
the CM sequence $\hat{S}=\{(x_i,v_i)\}_{i=0}^{k+1}$, and by the definition of $g$,
\[
g(y) \ge g(\hat{S},y)= \skal{y-x}{v}+\skal{x-x_k}{v_k}+\sum_{i=1}^k \skal{x_i-x_{i-1}}{v_{i-1}}.
\]
By \eqref{ineq_alpha}, 
\[
g(y) > \skal{y-x}{v}+ \alpha,
\]
which completes the proof.
\end{proof}
We obtain Theorem \ref{ThM} applying  Theorem \ref{th_BCC} for the inclusion
\[
\dot x(t) \in G(x(t)), \qquad x(0)=x_0, ,\quad t\in I=[0,T],
\]
with $G(x)$ defined in Theorem \ref{l1}.

\vspace{0.5em}
\begin{corollary}
Theorem \ref{ThM} holds if one replaces \textbf{A2} by \eqref{WMC2} 
for every finite sequence $\{x_i\}_{i=0}^m$.
\end{corollary}
\begin{remark}
As we mentioned above, in one dimension the weak monotonicity
coincides with the weak componentwise monotonicity for which the existence
of solutions is proved in \cite{FarDonBai:14}. We conjecture that there exist
solutions for weakly monotone right-hand sides also in higher dimensions.
\end{remark}

\vspace{0.5em}

\textbf{Acknowledgements.} 
The research  is 
partly supported by the Austrian Science Foundation (FWF) under grant
P 26640-N25 and by Minkowski Minerva 
Center for Geometry at Tel-Aviv University. 
The author is grateful to 
R. Baier, T. Donchev and especially to an anonymous referee for their valuable 
remarks.

\end{document}